\documentclass[leqno]{amsart}

\usepackage{amsmath}
\usepackage{amssymb}
\usepackage{enumerate}
\usepackage[mathscr]{eucal}

\newtheorem{theorem}{Theorem}[section]

\newtheorem{corollary}{Corollary}[theorem]
\newtheorem{lemma}{Lemma}[section]

\newtheorem{remark}{Remark}[section]
\newtheorem{example}{Example}[theorem]
\newtheorem*{acknowledgement}{\textnormal{\textbf{Acknowledgements}}}

\title[Birkhoff-James orthogonality and smoothness of bounded linear operators]
 {Birkhoff-James orthogonality and smoothness of bounded linear operators} 

\author{K. Paul, D. Sain  and P. Ghosh}

\newcommand{\acr}{\newline\indent}

\address{\llap{\,}Department of Mathematics\acr
                              Jadavpur University\acr
                              Kolkata 700032\acr
                              West Bengal\acr
                              INDIA}
\email{ kalloldada@gmail.com; saindebmalya@gmail.com; ghosh.puja1988@gmail.com}


\thanks{ Third author would like to thank CSIR, Govt. of India respectively  for the financial support.}

\begin{document}

\begin{abstract}
We present a sufficient condition for smoothness of  bounded linear operators on  Banach spaces for the first time. Let $T, A \in B(\mathbb{X}, \mathbb{Y}),$ where $\mathbb{X}$ is a real Banach space and $\mathbb{Y}$ is a real normed linear space. We find sufficient condition for $ T \bot_{B} A   \Leftrightarrow Tx \bot_{B} Ax $ for some $ x \in S_{\mathbb{X}}$ with $  \|Tx\| = \|T\|, $ and use it to show that $T$ is a smooth point in $ B(\mathbb{X}, \mathbb{Y}) $ if  $T$ attains its norm at unique (upto muliplication by scalar) vector $  x \in S_{\mathbb{X}},$ $Tx$ is a smooth point of $\mathbb{Y} $ and {\em sup}$_{y \in C} \|Ty\| < \|T\|$ for all  closed subsets $C$ of $S_{\mathbb{X}}$ with $d(\pm x,C) > 0.$ For operators on a  Hilbert space $ \mathbb{H}$ we show that $ T \bot_{B} A   \Leftrightarrow Tx \bot_{B} Ax $ for some $ x \in S_{\mathbb{H}}$ with $  \|Tx\| = \|T\| $ if and only if  the norm attaining set $M_T = \{ x \in S_{\mathbb{H}} : \|Tx\| = \|T\| \} = S_{H_0}$ for some finite dimensional subspace $H_0$ and $ \|T\|_{{H_o}^{\bot}} < \|T\|.$ We also characterize smoothness of compact operators on normed  spaces and  bounded linear operators on  Hilbert spaces.
\end{abstract}
\maketitle

\section{Introduction} 
Let $ (\mathbb{X},\|.\|) $ be a real normed  space. Let $ B_\mathbb{X}=\{x \in \mathbb{X} : \|x\| \leq 1\} $ and $ S_\mathbb{X}=\{x \in \mathbb{X} : \|x\|=1\} $ be the unit ball and the unit sphere of the normed space $ \mathbb{X}$ respectively. Let $ B(\mathbb{X},\mathbb{Y}) ( K(\mathbb{X},\mathbb{Y}))$ denote the set of all bounded (compact) linear operators from $ \mathbb{X} $ to another real normed  space $ \mathbb{Y}. $ We write $ B(\mathbb{X},\mathbb{Y}) = B(\mathbb{X}) $ and $ K(\mathbb{X},\mathbb{Y}) = K(\mathbb{X}) $ if $ \mathbb{X} = \mathbb{Y}. $ \\
$ T \in B(\mathbb{X},\mathbb{Y}) $ is said to attain its norm at $ x_{0} \in S_\mathbb{X} $ if $ \| Tx_{0} \| = \| T \|. $ Let $M_T$ denote the set of all unit vectors in $S_\mathbb{X}$ at which $T$ attains norm, i.e.,
\[ M_T = \{ x \in S_\mathbb{X} \colon \|Tx\| = \|T\| \}. \] 
  The notion of Birkhoff-James orthogonality \cite{B} plays a very important role in the geometry of Banach spaces. For any two elements $ x,y \in \mathbb{X}, $ $ x $ is said to be orthogonal to $ y $ in the sense of Birkhoff-James, written as $ x \bot_B y, $ if and only if  
\[ \|x\| \leq \|x+\lambda y\| ~ \forall \lambda \in \mathbb R. \]
Similarly for  $ T,A \in  \mathbb{B}(\mathbb{X}, \mathbb{Y}), $ $ T $ is said to be orthogonal to $ A,$ if and only if  
  \[ \|T\| \leq \|T+ \lambda A\| ~\forall \lambda \in \mathbb R. \] 
  An element  $0\neq  x \in \mathbb{X}$ is said to be a smooth point if there is a unique hyperplane $H$ supporting $B(0,\|x\|)$ at $x$. Equivalently $x$ is said to be a smooth point if there is a unique linear functional $f \in {\mathbb{X}}^{*} $ such that $\|f\|=1$ and $f(x) = \|x\|.$  From \cite{J} it follows that $x$ is a smooth point if and only if  $ x \bot_B y $ and $ x \bot_B z $ implies $ x \bot_B (y+z)$ i.e., if and only if  Birkhoff-James orthogonality is right additive at $x.$

\medskip In any  normed  space $\mathbb{X}$, if  $ x \in M_T$ with  $ Tx \bot_B Ax $ then $ T\bot_B A.$ The question that arises is when the converse is true i.e., if $ T\bot_B A$  then whether there exists $ x \in M_T$ such that  $  Tx \bot_B Ax .$   We  find  sufficient conditions for  $ T\bot_B A  \Leftrightarrow    Tx \bot_B Ax   $ for some $  x\in M_T.$   In Theorem $ 2.1 $ of  \cite{SP}  Sain and Paul proved  that if $ \mathbb{X} $ is a finite dimensional real normed  space and $M_T = D \cup(-D) $  ($ D $ is a connected subset of $ S_\mathbb{X} $) then for any $ A \in B(\mathbb{X}), T \bot_{B} A\Leftrightarrow  Tx \bot Ax$ for some $x \in M_T.$ In \cite{SPH} Sain et al. proved that if $T$ is a bounded linear operator on a normed  space $\mathbb{X}$  of dimension $2$ with $ T \bot_B A \Leftrightarrow Tx \bot_B Ax$ for some $ x \in M_T,$  then $M_T = D \cup(-D) ,$ where $ D $ is a connected subset of $ S_\mathbb{X} .$  

\medskip In this paper we prove that if $ \mathbb{X} $ is a reflexive Banach space and  $ T $ is a compact linear operator from $\mathbb{X}$ to $\mathbb{Y}$ with $M_T =  D \cup (-D) $ ($ D $ is a compact connected subset of $ S_{\mathbb{X}} $) then for any compact linear operator $A,$ $ T \bot_{B} A \Leftrightarrow  Tx \bot Ax$ for some $x \in M_T.$ This result substantially improves upon Theorem $ 2.1 $ of Sain and Paul \cite{SP}. Examples may be given to show that if $ T, A $ are bounded  instead of compact, then $ T \bot_{B} A $ does not ensure the existence of $ x \in M_T $ such that $ Tx \bot_{B} Ax. $ To get a result in this direction for bounded linear operators, we need to have certain additional condition(s) on $ T. $ We find sufficient conditions  for which $T \bot_{B} A \Leftrightarrow Tx \bot_{B} Ax $ for some $ x \in M_T. $

\medskip In case of Hilbert space we find  conditions  for $T\bot_B A  \Leftrightarrow  \langle Tx, Ax \rangle = 0 $ for some $x \in M_T.$    In a Hilbert space $ \mathbb{H}, $ Bhatia and \v{S}emrl \cite{BS} and Paul  \cite{P} independently proved that $ T \bot_B A $ if and only if there exists $ x_n \in M_T $  such that  $ \langle Tx_n, Ax_n\rangle  \longrightarrow 0. $  It follows then that if the Hilbert space $\mathbb{H}$  is finite dimensional,  $T \bot_B A \Leftrightarrow  \langle Tx, Ax \rangle = 0$ for some $  x \in M_T.$  
In case of an infinite dimensional Hilbert space $\mathbb{H}$, examples can be given to show that $T\bot_B A$ but $M_T = \emptyset $ and so the question of whether $\langle Tx, Ax\rangle =0$ for $x \in M_T$ does not arise. Even if $M_T \neq \emptyset$, there are operators $T,A$ such that $T\bot_B A$ but there does not exist $x \in M_T$ with $\langle Tx, Ax \rangle =0.$ This implies that for such a result to be true in an infinite dimensional Hilbert space, we need to impose certain additional condition(s) on $T$. We prove that  $~T\bot_B A$ for $ A \in B(\mathbb{H})\Leftrightarrow Tx \bot Ax$ for some $x \in M_T$ if and only if  $M_T = S_{H_0}$and $\|T\|_{{H_0}^\bot} < \|T\|,$  where $H_0$ is a finite dimensional subspace of $\mathbb{H}.$

\medskip As an application of these results on Birkhoff-James orthogonality and operator norm attainment, we prove certain characterizations of smoothness of operators. This is a classical area of research in the geometry of Banach spaces and has been studied in great detail by several mathematicians including Holub \cite{HO}, Heinrich \cite{HR}, Hennefeld \cite{HF}, Abatzoglou \cite{A}, Kittaneh and Younis \cite{KY}. 
Smoothness of bounded linear operators on some particular spaces like $\ell^p$ spaces, etc. have been studied by Werner \cite{W} and Deeb and Khalil \cite{DK}.   Although characterization of smoothness of compact linear operators on normed  spaces have been obtained,  there is no  such result for bounded linear operators on a general normed space.  To the best of our knowledge, this is for the first time that a sufficient condition for smoothness of a bounded linear operator defined on a Banach space is being presented. We show that if $\mathbb{X}$ is a Banach space then $ T \in B(\mathbb{X}, \mathbb{Y}) $ is smooth if  $ T $ attains its norm only on $ \pm x \in S_{\mathbb{X}}, $ $ Tx $ is  a smooth point of $ \mathbb{Y} $ and {\em sup}$_{y \in C, \|y\|=1} \|Ty\| < \|T\|$ for all  closed subsets $C$ of $S_{\mathbb{X}}$ with $d(\pm x,C) > 0$. We use the result of \cite{P} to prove that  if $T$ is a bounded linear operator on a Hilbert space $\mathbb{H}$ then $T$  is  smooth if and only if $T$ attains norm only at $\pm x$ and {\em sup}$\{\|Ty\|: x\bot y, y\in S_\mathbb{H}\} < \|T\|$.

\medskip 


\section{Operator norm attainment and Birkhoff-James orthogonality in a Banach space}
\noindent We first prove that if $T$ is a compact linear operator on a reflexive Banach space $ \mathbb{X} $ with $M_T = D \cup (-D)$ ($D$ is a non-empty compact connected subset of $S_{\mathbb{X}}$), then for any $A \in K(\mathbb{X},\mathbb{Y}),$ $T \bot_B A \Leftrightarrow Tx \bot_B Ax$ for some $x \in M_T.$ We begin with the following lemma.
\begin{lemma}\label{lemma:lambda}
Let $T \in B(\mathbb{X},\mathbb{Y})$ and $M_T = D \cup (-D)$ ($D$ is a non-empty compact connected subset of $S_{\mathbb{X}}$). Then for any $A \in B(\mathbb{X},\mathbb{Y}), $ either  there exists $x \in M_T$ such that $Tx \bot_B Ax$ or there exists $ \lambda_0 \neq 0 $ such that 
$ \| Tx + \lambda_0Ax\| < \|Tx\|~ \forall x \in M_T.$
\end{lemma}
\begin{proof} 
Suppose that there exists no $x \in M_T$ such that $Tx \bot_B Ax$. 
  Let \[   W_1=\{ x \in D : \|Tx+\lambda_x Ax\| < \|T\| ~\mbox{for some}~ \lambda_x > 0\}\]
\[\mbox{and}~  W_2=\{x \in D : \|Tx+\lambda_x Ax\| < \|T\| ~\mbox{for some}~ \lambda_x < 0\}.\] 
Then it is easy to check that both $W_1,W_2$ are open sets in $D$ and $ D = W_1 \cup W_2 $. The connectedness of $D$ ensures that either $D = W_1$ or $D = W_2.$

 Consider the case $D = W_1.$   Then for each $ x \in D,$ there exists $\lambda_x \in(0,1) $ such that $ \|Tx + \lambda_x Ax \| < \|Tx\| = \|T\|.$ 
By the convexity of the norm function it now follows that 
\[ \|Tx + \lambda Ax \| < \|Tx\| = \|T\| ~ \forall ~ \lambda \in (0, \lambda_x). \] 
\noindent We consider the continuous function $ g : S_X \times [ -1 , 1] \longrightarrow  \mathbb{R} $ defined by 
$$g(x,\lambda) = \|Tx + \lambda Ax\|.$$
 We have $ g(x, \lambda_x) = \| Tx +  \lambda_x Ax \| < \|T\|$ and so by the continuity of $g$, there exists $r_x , \delta_x > 0$ such that 
 $ g(y, \lambda) < \|T\| ~~~ \forall ~~~ y \in B(x, r_x) \cap S_{\mathbb{X}} ~\mbox{ and}~ \forall ~ \lambda \in ( \lambda_x - \delta_x ,  \lambda_x + \delta_x).$ 
Let $ y \in B( x, r_x) \cap D $. Then for any $ \lambda \in (0, \lambda_x ) $ we get
\begin{eqnarray*}
& ~~~ & Ty + \lambda Ay = \Big(1 - \frac{\lambda}{\lambda_x}\Big) Ty + \frac{\lambda}{\lambda_x} (Ty + \lambda_x Ay )\\
& \Rightarrow & \| Ty + \lambda Ay \| < (1 - \frac{\lambda}{\lambda_x})\|T\| + \frac{\lambda}{\lambda_x} \|T\| \\
& \Rightarrow & \| Ty + \lambda Ay \| < \| T \|
\end{eqnarray*}
Therefore $  g(y,\lambda)  < \|T\| ~~~ \forall ~~~ y \in B(x, r_x) \cap D $ and $ \forall \lambda \in (0, \lambda_x).$ \\
\noindent Consider the open cover $ \{ B(x, r_x) \cap D : x \in D \} $   of $ D.$ By the compactness of $D$, this cover has a finite subcover $\{ B(x_i,r_{x_i}) \cap D : i=1,2,\ldots,n\} $ so that  
  \[ D \subset \cup_{i=1}^{n} B(x_i,r_{x_i}) .\]
\noindent Choose $ \lambda_0 \in \cap_{i=1}^{n} (0, \lambda_{x_{i}}). $
Then for any $x \in M_T$, $\|Tx+\lambda_0 Ax\| < \|T\|$. \\
If $ D=W_2$ then similarly we can show that there exists some $ \lambda_0 < 0 $ such that 
for any $x \in M_T$, $\|Tx+\lambda_0 Ax\| < \|T\|$. \\
This completes the proof of lemma.
\end{proof}
\begin{theorem}\label{theorem:compact operator}
Let $\mathbb{X}$ be a reflexive Banach space and $\mathbb{Y}$ be any normed space. Let $T \in K(\mathbb{X},\mathbb{Y})$ and $M_T = D \cup (-D)$ ($D$ is a non-empty compact connected subset of $S_{\mathbb{X}}$). Then for any $A \in K(\mathbb{X},\mathbb{Y}), T \bot_B A $ if and only if  there exists $x \in M_T$ such that $Tx \bot_B Ax$.
\end{theorem}
\begin{proof} 
Suppose there exists no $x \in M_T$ such that $Tx \bot_B Ax$. 
Then by applying Lemma~\ref{lemma:lambda}, we get some $ \lambda_0 \neq 0$ such that 
$$\|Tx+\lambda_0 Ax\| < \|T\|~\forall x \in M_T.$$
Without loss of generality we assume that $ \lambda_0 > 0.$ 

\noindent For each $ n \in \mathbb{N}, $ the  operator $(T+\frac{1}{n} A),$  being compact on a reflexive normed space, attains its norm. So there exists $  x_n \in S_\mathbb{X}$ such that $\|T+\frac{1}{n} A\| = \|(T+\frac{1}{n} A)x_n\|$.

Now $\mathbb{X}$ is reflexive and  so $ B_{\mathbb{X}} $ is weakly compact,  hence we can find a subsequence $\{x_{n_k}\}$ of $\{x_n\}$ such that $ x_{n_k} \rightharpoonup x_0 $ (say) in $ B_{\mathbb{X}} $  weakly. Without loss of generality we assume that $ x_{n} \rightharpoonup x_0 $ weakly. Then $T, A$ being compact, $Tx_{n} \longrightarrow Tx_0$ and $Ax_{n} \longrightarrow Ax_0.$  
As $ T \bot_B A $  we have $\|T+\frac{1}{n} A\| \geq \|T\|  ~\forall n \in \mathbb{N}$ and so $ \|Tx_{n} + \frac{1}{n} Ax_{n}\|  \geq \|T\| \geq \|Tx_n\|~\forall n \in \mathbb{N}$. Letting $ n \longrightarrow \infty $ we get $ \|Tx_0\|  \geq \|T\| \geq \|Tx_0\| .$  Then $ x_0 \in M_T$. 

We finally show that $ Tx_0 \bot_B Ax_0.$ \\
For any $ \lambda > \frac{1}{n}$ we claim that  $\|Tx_n+\lambda Ax_n\| \geq \|Tx_n\|$. Otherwise
\begin{eqnarray*}
Tx_n+\frac{1}{n} Ax_n & = & \Big(1 - \frac{1}{n \lambda}\Big) Tx_n + \Big(\frac{1}{n \lambda}\Big)(Tx_n + \lambda Ax_n) \\
\Rightarrow  \|Tx_n+\frac{1}{n} Ax_n\| & < & \Big(1 - \frac{1}{n \lambda}\Big) \|Tx_n\| + \Big(\frac{1}{n \lambda}\Big) 
                                                                                                           \|Tx_n \|\\		\Rightarrow  \|Tx_n+\frac{1}{n} Ax_n\|  & < & \|Tx_n\|~ \mbox{, a contradiction.}
\end{eqnarray*}
Choose $\lambda >0 $. Then there exists $ n_0 \in \mathbb{N}$ such that $ \lambda > \frac{1}{n_0} $ and so for all $ n \geq n_0$ we get, $$\|Tx_n+\lambda Ax_n\| \geq \|Tx_n\|.$$
Letting $ n \longrightarrow \infty $ we get 
$$\| Tx_0 + \lambda Ax_0 \| \geq \|Tx_0\| \ldots \ldots \ldots (i)$$
We next show that  $\| Tx_0 + \lambda Ax_0 \| \geq \|Tx_0\|$ for each $ \lambda < 0.$ Suppose there exist some $ \lambda_1 < 0 $ such that $\| Tx_0 + \lambda_1 Ax_0 \| < \|Tx_0\|$. 
We already have $\| Tx_0 + \lambda_0 Ax_0 \| < \|Tx_0\|$. Then 
\begin{eqnarray*}
 T x_0 & = & ( 1 - \frac{\lambda_0}{\lambda_0 - \lambda_1})( Tx_0 + \lambda_0 Ax_0 ) + (\frac{\lambda_0}{\lambda_0 - \lambda_1})( Tx_0 + \lambda_1 Ax_0 ) \\
\Rightarrow \| Tx_0\| & < & ( 1 - \frac{\lambda_0}{\lambda_0 - \lambda_1}) \| Tx_0 \| +  (\frac{\lambda_0}{\lambda_0 - \lambda_1}) \|Tx_0\| \\
\Rightarrow \|Tx_0\| & < & \|Tx_0\|  \mbox{, a contradiction} 
\end{eqnarray*}
Thus $\| Tx_0 + \lambda Ax_0 \| \geq \|Tx_0\|$ for each $ \lambda < 0.$
This along with (i) shows that $ Tx_0 \bot_B Ax_0$. This completes the proof of the theorem.
\end{proof}
\begin{corollary}\label{corollary:compact*}
 Let $T \in K(\mathbb{X},\mathbb{Y}) $ and $ M_{T^{*}} = D \cup(-D)$  ($D$ is a compact connected subset of $S_{\mathbb{Y}^*}$). Then for any $A \in K(\mathbb{X},\mathbb{Y}), T \bot_B A $ if and only if  there exists $g \in M_{T^*}$ such that $T^*g \bot_B A^*g$.
\end{corollary}
\begin{proof} 
Noting that $T \bot_B A $ if and only if  $ T^* \bot_B A^* $ and $S_{\mathbb{Y}^*}$ is weak$^*$ compact we can apply the above Theorem~\ref{theorem:compact operator} to conclude that if $ T \bot_B A $ then there exists $g \in M_{T^*}$ such that $T^*g \bot_B A^*g$. The other part is obvious.
\end{proof}
\begin{remark} Theorem 2.1 of Sain and Paul \cite{SP} is a simple consequence of the above Theorem~\ref{theorem:compact operator}, since every finite dimensional normed  space is reflexive and every linear operator defined there is compact. 
\end{remark}
\noindent The following example shows that the above theorem can not be extended to bounded linear operators without any additonal restriction on $T.$ 
\begin{example}
Consider $T \colon \ell_2 \to \ell_2$ defined by
$Te_1 = -e_1$, and $Te_n = (1 - 1/n)e_n$ for $n \geq 2$,
where $\{e_n \colon n \in \mathbb{N}\}$ is the usual orthonormal basis
for the Hilbert space $\ell_2$.  Then $T$ attains norm only at
$\pm e_1$.  Let $A=I$, the identity operator on $\ell_2$.
It is easy to check that $T \bot_B A$. Indeed,
$\|(T + \lambda A)e_1\| \geq \|T\|$ for all $\lambda \leq 0$, and
$\|(T + \lambda A)e_n\| \geq \|T\|$, for all $\lambda \geq 1/n$.
But $Te_1$ is not orthogonal to $Ae_1$ in the sense of Birkhoff-James.
\end{example}
\noindent In the next theorem we  consider $T$ with an additional condition that {\em sup}$\{\|Tx\|:  x\in C\} < \|T\|$ for all  closed subset $C$ of $S_{\mathbb{X}}$ with $d(M_T, C) > 0.$ 
\begin{theorem}\label{theorem:compact suff}
Let $\mathbb{X}$ be a Banach space,  $T \in B(\mathbb{X},\mathbb{Y}),~M_T = D \cup (-D)$ ($D$ is a non-empty compact connected subset of $S_{\mathbb{X}}$). If {\em sup}$\{\|Tx\|:  x\in C\} < \|T\|$ for all  closed subset $C$ of $S_{\mathbb{X}}$ with $d(M_T, C) > 0$ then for any $ A \in B(\mathbb{X},\mathbb{Y}),~ T\bot_B A, $ if and only if  there exists $z \in M_T$ such that $Tz \bot_B Az$.
\end{theorem}
\begin{proof} 
Assume that $M_T = D \cup (-D)$ ($D$ is a non-empty compact connected subset of $S_{\mathbb{X}}$) and $\sup\{\|Tx\|:  x\in C\} < \|T\|$ for all  closed subset $C$ of $S_{\mathbb{X}}$ with $d(M_T, C) > 0.$ \\
If $z \in M_T$ such that $Tz \bot_B Az$ then clearly $T \bot_B A.$ For the other part, suppose that  $ T \bot_B A $ but there exists no $x \in M_T$ such that $Tx \bot_B Ax$. 
Then by applying  Lemma~\ref{lemma:lambda}, we get some $ \lambda_0 \neq 0 $ such that 
$$\|Tx+\lambda_0 Ax\| < \|T\|~\forall x \in M_T.$$
Without loss of generality we assume that $ \lambda_0 > 0.$ 

Now $ x \longrightarrow \| Tx+\lambda_0 Ax\| $ is a real valued continuous function from $S_{\mathbb{X}}$ to $\mathbb{R}$.  As  $M_T$ is a compact subset of $S_{\mathbb{X}}$ so this function attains its maximum on $M_T$.   Then we can find an $ \epsilon_1 > 0 $ such that 
$$\|Tx+\lambda_0 Ax\| < \|T\| - \epsilon_1 ~\forall x \in M_T.$$ 
Choose $ \epsilon_x = ||T\| - \epsilon_1 - \|Tx+\lambda_0 Ax\|. $ For each $x \in M_T$ we have $\epsilon_x > 0 $ and so by continuity of the function $ T + \lambda_0 A$ at the point $ x $ we can find an open ball $ B(x, r_x) $ such that 
$$ \|(T+\lambda_0 A)(z-x)\| < \epsilon_x \forall z \in B(x, r_x) \cap S_{\mathbb{X}}.$$
Then $\|(T+\lambda_0 A)z\| < \|T \|- \epsilon_1 \forall z \in B(x, r_x) \cap S_{\mathbb{X}}.$

Again let $\lambda \in(0, \lambda_0)$.	Then for all $ z \in B(x,r_x) \cap S_{\mathbb{X}} $
\begin{eqnarray*}
	Tz+\lambda Az &= & (1-\frac{\lambda}{\lambda_0}) Tz + \frac{\lambda}{\lambda_0} (Tz+\lambda_0 Az)\\
	\Rightarrow  \|Tz+\lambda Az\| &\leq & (1-\frac{\lambda}{\lambda_0}) \|Tz\| + \frac{\lambda}{\lambda_0} \|Tz+\lambda_0 Az\|\\
	\Rightarrow  \|Tz+\lambda Az\| & < &  (1-\frac{\lambda}{\lambda_0}) \|T\| + \frac{\lambda}{\lambda_0}(\|T\|-\epsilon_1) \\
	                               & = & \|T\|-\frac{\lambda}{\lambda_0} \epsilon_1
	\end{eqnarray*}
The compactness of $M_T$ ensures that the cover $ \{ B(x,r_x) \cap M_T : x \in M_T\} $ has a finite subcover $\{ B(x_i,r_{x_i}) \cap M_T : i =1,2,\ldots n\} $ so that 
$$ M_T \subset \cup_{i=1}^{n} B(x_i,r_{x_i}).$$
So for $\lambda \in(0, \lambda_0)$ and $ z \in \Big(\cup_{i=1}^{n} B(x_i,r_{x_i}) \Big) \cap S_{\mathbb{X}}$ we get 	
$$ \|Tz+\lambda Az\| < 	\|T\|-\frac{\lambda}{\lambda_0} \epsilon_1.$$
Consider $ C = \cap_{i=1}^{n} B(x_i,r_{x_i}) ^{c}$. Then $C$ is a closed subset of $ S_{\mathbb{X}}$ with $C \cap M_T = \emptyset.$ As $M_T $ is compact so $ d(C,M_T) > 0.$ By the hypothesis {\em sup}$\{\|Tz\|: z \in C\} < \|T\|$ and so there exists $\epsilon_2 > 0$ such that  {\em sup}$\{\|Tz\|:  z\in C\} < \|T\| - \epsilon_2$.

Choose $ 0 < \widetilde{\lambda} < \min\{\lambda_0, \frac{\epsilon_2}{2\|A\|}\}$. Then for all $ z \in C $ we get 
\begin{eqnarray*}
\|Tz+\widetilde{\lambda} Az\| &\leq & \|Tz\|+|\widetilde{\lambda}|\|Az\| \\
                              & <  & \|T\|-\epsilon_2 +|\widetilde{\lambda}|\|A\|\\
															& < & \|T\| -\frac{1}{2}\epsilon_2
\end{eqnarray*}

Choose $\epsilon = \min \{\frac{1}{2}\epsilon_2, \frac{\widetilde{\lambda}}{\lambda_0} \epsilon_1\}$. Then for all $x \in S_{\mathbb{X}} $ we get $$\|Tx+\widetilde{\lambda} Ax\| < \|T\|-\epsilon.$$
This shows that $ \|T+\widetilde{\lambda} A\| < \|T\|$, which contradicts the fact that $ T \bot_B A.$ 
This completes the proof.
\end{proof}

\section{Operator norm attainment in a Hilbert space $\mathbb{H}$ and Birkhoff-James orthogonality in $B(\mathbb{H})$}

From \cite{BS,P} it follows that  if $T $ is a bounded linear operator on a finite dimensional Hilbert space  then $ T \bot_B A$ for  $A \in B(\mathbb{H}) $ if and only if  there exists  $x\in M_T$ such that $ \langle Tx, Ax \rangle = 0.$ The result is not true if the space is of infinite dimension. In the following theorem we  settle the problem for  Hilbert space of any dimension.

\begin{theorem} \label{theorem:Hilbert bounded}Let $T \in B(\mathbb{H}).$  Then  for any $A \in B(\mathbb{H}),~T\bot_B A \Leftrightarrow Tx_0 \bot Ax_0$ for some $x_0 \in M_T$ if and only if  $M_T = S_{H_0},$ where $H_0$ is a finite dimensional subspace of $\mathbb{H}$ and $\|T\|_{{H_0}^\bot} < \|T\|.$ 
\end{theorem}

\begin{proof} Without loss of generality we assume that $ \|T\|= 1.$ We first prove the necessary part. From Theorem 2.2 of Sain and Paul \cite{SP} it follows that in case of a Hilbert space the norm attaining set $M_T$ is always a unit sphere of some subspace of the space. We  first show that the subspace is finite dimensional. Suppose $M_T$ be the unit sphere of an infinite dimensional subspace $H_0.$ Then we can find a set $\{e_n : n \in \mathbb{N} \}$ of orthonormal   vectors in $H_0.$  Extend the set to a complete orthonormal basis $\mathcal{B} = \{ e_{\alpha} : \alpha \in \Lambda \supset \mathbb{N} \} $ of $\mathbb{H}.$
For each $e_\alpha \in H_0 \cap \mathcal{B} $ we have 
$$ \| T^{*}T \| = \|T\|^2 = \|Te_\alpha \|^2 = \langle T^{*}T e_\alpha , e_\alpha \rangle \leq \|T^{*}Te_\alpha \| \|e_\alpha \| \leq \|T^{*}T\| $$
so that by the equality condition of Schwarz's inequality we get $ T^{*}Te_\alpha = \lambda_\alpha e_\alpha $ for some scalar $\lambda_\alpha.$ 
Thus $\{ Te_\alpha : e_\alpha \in H_0 \cap \mathcal{B}\}$ is a set of orthonormal vectors in $\mathbb{H}.$ 
 Define $A : \mathcal{B} \longrightarrow \mathbb{H} $ as follows : 
\begin{eqnarray*}
           A(e_n) & =   & \frac{1}{n^2} Te_n, n \in \mathbb{N} \\
         	 A(e_\alpha) & = & Te_\alpha,~ e_\alpha \in  H_0 \cap \mathcal{B} - \{e_n : n \in \mathbb{N} \}\\
					 A(e_\alpha) & = & 0,~e_\alpha \in \mathcal{B} - H_0 \cap \mathcal{B}
\end{eqnarray*}	

As $\{ Te_\alpha : e_\alpha \in H_0 \cap \mathcal{B}\}$ is a set of orthonormal vectors in $\mathbb{H}$ it is easy to see that  $A$ can be extended as  a bounded linear operator on $\mathbb{H}.$

Now for  any scalar $\lambda , ~ \| T + \lambda A \| \geq \| (T + \lambda A)e_n \| = \| ( 1 + \frac{\lambda}{n^2})Te_n \| = \mid 1 +  \frac{\lambda}{n^2} \mid \|T\| \longrightarrow \|T\|.$ Thus $ T \bot_{B} A.$ 

We next show that there exists no $ x \in M_T$ such that $ Tx \bot_B Ax.$ 
Let $ x = \sum_{\alpha} \langle x, e_\alpha \rangle e_\alpha~ \in M_T.$  Then 
$$ \langle Tx, Ax \rangle = \sum_n \frac{1}{n^2} \mid \langle x, e_n \rangle\mid ^2	\|T\|^2+ \sum_{\alpha \notin \mathbb{N}} \mid \langle x, e_\alpha \rangle\mid ^2 	\|T\|^2 $$
and so $ \langle Tx, Ax \rangle =  0 $ if and only if  $x = 0.$
Thus $T \bot_B A$ but there exists no $x \in M_T$ such that $ Tx \bot_B Ax.$ This is a contradiction and so $H_0$ must be finite dimensional.

We next claim that $\|T\|_{{H_0}^{\bot}} < \|T\|.$  Suppose $\|T\|_{{H_0}^{\bot}} = \|T\|.$  As $T$ does not attain its norm on ${H_0}^{\bot}$ and  $ \|T\| $= {\em sup} $\{ \|Tx\| : x \in S_{{H_0}^{\bot}} \} $  	there exists $\{e_n\} $ in ${H_0}^{\bot}$ such that $\|Te_n\| \longrightarrow \|T\|.$  We have $ \mathbb{H} = H_0 \oplus {H_0}^{\bot} $.

Define $A : \mathbb{H} \longrightarrow \mathbb{H} $ as follows:
\[ Az = Tx, \mbox{ where } z = x + y, x \in H_0, y \in {H_0}^{\bot} \]
 Then it is easy to check $A$ is bounded on $\mathbb{H}$. Also for any scalar $ \lambda $, $ \| T + \lambda A \| \geq \| (T + \lambda A) e_n \|  = \| Te_n \| $ holds for each $ n \in N.$ Then $ \| T + \lambda A \| \geq \|T\| $ for all $ \lambda $ so that $ T \bot_B A.$ 
But there exists no $ x\in M_T$ such that $\langle Tx, Ax\rangle = 0.$ This contradiction completes the necessary part of the theorem.

We next prove the sufficient part. If $ \langle  Tx_0 , Ax_0 \rangle = 0 $ for some $x_0 \in M_T,$ then $ T \bot_B A.$ Next let $T\bot_B A .$ Then by  Paul  \cite{P}   there exists  $ \{z_n\}\subset S_\mathbb{H} $ such that $\|Tz_n\|\rightarrow \|T\|$ and $\langle Tz_n,Az_n \rangle \longrightarrow 0$.
For each $n \in \mathbb{N} $ we have $z_n =  x_n + y_n,$  where $x_n \in H_0,  y_n~\in {H_0}^{\bot}$. 
Then $ \|z_n\|^2 = 1 = \|x_n\|^2 + \|y_n\|^2 $ and so $\|x_n\| \leq 1, ~\forall n \in \mathbb{N}.$ 
As $H_0$ is a finite dimensional subspace so $\{x_n\},$ being bounded, has a convergent subsequence converging to some element of $H_0.$ Without loss of generality we assume that $ x_n \longrightarrow x_0 $ (say) in $H_0$ in norm. Now for each non-zero element $x \in H_0$ we have,
 $$ \|T^*T\|\|x\|^2 \leq \|T\|^2 \|x\|^2 = \| Tx\|^2 = \langle T^*Tx,x \rangle \leq \|T^*Tx\| \|x\| \leq \|T^*T\|\|x\|^2 $$ and so 
$ \langle T^*Tx,x \rangle = \|T^*Tx\| \|x\|.$ By the equality condition of Schwarz's inequality  $T^*Tx = \lambda_x x$ for some $\lambda_x.$ 

Now $ \langle T^*Tx_n,y_n \rangle =  \langle T^*Ty_n,x_n \rangle = 0 $ and so 
\begin{eqnarray*}
\langle T^*Tz_n, z_n \rangle & = &  \langle T^*Tx_n,x_n \rangle  + \langle T^*Tx_n,y_n \rangle + \langle T^*Ty_n,x_n \rangle  +  \langle T^*Ty_n, y_n \rangle  \\
\Rightarrow \lim \|Tz_n\|^2 &  = &  \lim \|Tx_n\|^2 + \lim \|Ty_n\|^2 .\\
\Rightarrow \|T\|^2 &  = &  \|Tx_0\|^2 + \lim \|Ty_n\|^2 \\
\Rightarrow  \lim\|Ty_n\|^2 & = &  \|T\|^{2} (1- \|x_0\|^2) \\
\Rightarrow \lim\|Ty_n\|^2 & = & \|T\|^{2} \lim \|y_n\|^2 \ldots\ldots(1)
\end{eqnarray*}
By hypothesis  {\em sup}$\{\|Ty\|:  y \in {H_0}^{\bot},~\|y\|=1\} < \|T\| $ and so by (1) there does not exist any non-zero subsequence of $\{\|y_n\|\}.$ So we conclude $y_n = 0  \forall~ n $ and $ z_n = x_n  \forall ~n$.
Then $ \langle Tz_n, Az_n \rangle \rightarrow 0 \Rightarrow  \langle Tx_0,Ax_0 \rangle = 0 $. This completes the proof.  
 
\end{proof}

\section{Smoothness of bounded linear operators}

\noindent As an application of the results obtained in the previous sections we first give sufficient condition for smoothness of compact  linear operators on a Banach space. Later on we  give sufficient condition for   smoothness of bounded linear operators on a Banach space. 
\begin{theorem}\label{theorem:smooth compact sufficient} Let $\mathbb{X}$ be a reflexive Banach space and $\mathbb{Y}$ be a normed  space. Then $T \in K(\mathbb{X},\mathbb{Y})$ is smooth if $T$ attains norm at a unique (upto scalar multiplication) vector $x_0$ (say)  of $S_{\mathbb{X}}$ and $ Tx_0$ is a smooth point.
\end{theorem}
\begin{proof} Assume $T$ attains norm at a unique (upto scalar multiplication) vector $x_0$ (say)  of $S_{\mathbb{X}}$ and $ Tx_0$ is a smooth point. 
We show that for any $P, Q \in K(\mathbb{X},\mathbb{Y}),$  if $ T \bot_B P $ and $ T \bot_B Q$ then $ T \bot_B ~(P + Q).$  
By Theorem~\ref{theorem:compact operator}, we get $ Tx_0 \bot_B Px_0 $ and $ Tx_0 \bot_B Qx_0.$ As $Tx_0$ is a smooth point so we get $ Tx_0 \bot_B~( Px_0 + Qx_0). $ Then $ T \bot_B ~(P + Q).$\\
This completes the proof.
\end{proof}
\begin{remark} This improves  the result [Theorem 2.2 ] proved by Hennefeld  \cite{HF} in which the author assumed $\mathbb{X}$ to be a smooth reflexive Banach space with a Schauder basis. 
\end{remark}
Conversely we show that the conditions are necessary.
\begin{theorem}\label{theorem:smooth compact necessary} Let $\mathbb{X}$ be a reflexive Banach space and $\mathbb{Y}$ be a normed  space. If $T \in K(\mathbb{X},\mathbb{Y})$ is smooth then $T$ attains norm at a unique (upto scalar multiplication) vector $x_0$ (say)  of $S_{\mathbb{X}}$ and $ Tx_0$ is a smooth point.
\end{theorem}
\begin{proof} Since the space $\mathbb{X}$ is reflexive and $T$ is compact so there exists $x \in S_{\mathbb{X}} $ such that $\|Tx\| = \|T\|.$ We show that if $ \| Tx_1 \| = \|Tx_2 \| = \|T\|$ for $ x_1, x_2 \in S_{\mathbb{X}} $ then $ x_1 = \pm x_2.$ If possible let $ x_1 \neq \pm x_2.$ There exists a subspace $ H_1 $ of codimension 1 such that $ x_1 \bot_{B} H_1.$ There exists a scalar $a$ with $\mid a \mid \leq 1 $ such that $ax_1 + x_2 \in H_1.$ Again there exists a subspace $H_2$ of $H_1$ with codimension 1 in $H_1$ such that $ (ax_1 +x_2) \bot_{B} H_2.$ Now every element $ z \in S_{\mathbb{X}} $ can be written uniquely as $ z = \alpha x_1 + h_1 $ for some scalar $\alpha $ and $ h_1 \in H_1.$ Again $h_1$ can be written uniquely as $h_1 = \beta(ax_1 + x_2) + h_2$ for some scalar $\beta $ and $h_2 \in H_2.$  Thus $ z = ( \alpha + a \beta ) x_1 + \beta x_2 + h_2 .$ Define operators $ A_1, A_2 : \mathbb{X} \longrightarrow \mathbb{Y} $ as follows : \\
\[ A_1(z) = (\alpha + a \beta) Tx_1 , ~A_2(z) = \beta T x_2 +  T h_2.\]
Clearly both $A_1,A_2 $ are compact linear operators. Then $ T \bot_{B} A_1, T \bot_{B}A_2 $ but $ T = A_1 + A_2 $ which shows that $ T $ is not orthogonal to $ A_1 + A_2$ in the sense of Birkhoff-James. This shows that $T$ is not smooth. Hence $T$ attains norm at unique( upto scalar multiplication )  vector $x_0 \in S_{\mathbb{X}} $. \\

We next show that $Tx_0$ is a smooth point in $\mathbb{Y}.$ If possible let $Tx_0$ be not smooth. Then there exists $y,z \in \mathbb{Y} $ such that $ Tx_0 \bot_{B} y, Tx_0 \bot_{B} z $ but $ Tx_0$ is not orthogonal to $ y +z$ in the sense of Birkhoff-James.  There exists a hyperplane $H$ such that $ x_0 \bot_{B} H.$ Define two operators $A_1, A_2 : \mathbb{X} \longrightarrow \mathbb{Y} $ as follows : 
\[ A_1( ax_0 + h) = a y, ~ A_2(ax_0 + h) = az. \]
Then it is easy to check that both $A_1, A_2$ are compact linear operators and $ T \bot_{B} A_1,$ $  T \bot_{B} A_2. $ 
But $T$ is not orthogonal to $A_1 + A_2$, otherwise since $M_T = \{ \pm x_0\} ,$ we have by 
Theorem~\ref{theorem:compact operator}, $ Tx_0 \bot_{B} (y + z) $, which is not possible. This contradiction shows that $Tx_0$ is a smooth point. 

\end{proof}

\begin{corollary}\label{corollary:compact**}
 $T \in K(\mathbb{X},\mathbb{Y}) $ is a smooth point if and only if  $T^{*}$ attains norm at a unique (upto scalar multiplication) vector $g $ (say) of $S_{\mathbb{Y}^*}$ and $T^*g$ is a smooth point. 
\end{corollary} 
\begin{proof} We first note that $T$ is smooth if and only if  $T^{*}$ is smooth. 
Then by using Corollary~\ref{corollary:compact*} and following the same method as above we can show $T$ is  a smooth point if and only if   $T^{*} $ attains norm at a unique (upto scalar multiplication) vector $g $ (say) of $S_{\mathbb{Y}^*}$ and $T^*g$ is a smooth point. 
\end{proof}
\begin{remark}
In \cite{HR} Heinrich proved  necessary and sufficient conditions for smoothness of compact operators from a Banach space $ \mathbb{X} $ to a Banach space $ \mathbb{Y} $ using differentiabilty of the norm of a Banach space. In Theorem~\ref{theorem:smooth compact sufficient}, Theorem~\ref{theorem:smooth compact necessary} and Corollary~\ref{corollary:compact**}, we have given  alternative proofs of the   results without assuming $\mathbb{Y}$ to be a Banach space.
\end{remark}
\noindent We next give a sufficient condition for a bounded linear operator to be smooth.
\begin{theorem}
Let $\mathbb{X}$ be a Banach space and $\mathbb{Y}$ be a normed  space.  Then $T \in B(\mathbb{X},\mathbb{Y})$ is a smooth point if $T$ attains norm only at $\pm x_0$, $Tx_0$ is smooth and {\em sup}$\{\|Tx\|:  x\in C\} < \|T\|$ for all closed subsets $C$ of $ S_\mathbb{X}$ with $d( \pm x_0, C) > 0$.
\end{theorem}
\begin{proof} 
 Suppose $T$ attains norm only at $\pm x_0$, $Tx_0$ is smooth and {\em sup}$ \{\|Tx\|:  x\in C\} < \|T\|$ for all closed subsets $C$ of $ S_\mathbb{X}$ with $d(\pm x_0, C) > 0$.  
Let $ T \bot_B A_1,~ T \bot_B A_2 .$ Then by Theorem~\ref{theorem:compact suff}, $Tx_0 \bot_B A_1 x_0,~Tx_0 \bot_B A_2x_0.$ As $Tx_0$ is a smooth point so $Tx_0 \bot_B (A_1 + A_2)x_0 $ and so $ T \bot_B (A_1 + A_2).$ Thus $T$ is a smooth point. 

\end{proof}

We next prove the following: 
\begin{theorem}\label{theorem:bounded complmnt attain}
Let $\mathbb{X}$ be a Banach space and $\mathbb{Y}$ be a normed  space.  If $T \in B(\mathbb{X},\mathbb{Y})$ is a smooth point that attains norm only at $ \pm x_0 \in S_{\mathbb{X}} $ then {\em sup}$_{x \in H \cap S_{\mathbb{X}}} \|Tx\| < \|T\| $ where $H$ is a hyperplane such that $ x_0 \bot_{B} H.$ 
\end{theorem}
\begin{proof}
If possible let {\em sup}$_{x \in H \cap S_{\mathbb{X}}} \|Tx\| = \|T\| .$ Then there exists $ \{x_n\} \subset H \cap S_{\mathbb{X}} $ such that $ \|Tx_n\| \longrightarrow \|T\|.$  Every element $ z \in S_{\mathbb{X}} $ can be written as $ z = \alpha x_0 + h $ for some scalar $\alpha $ and $ h \in H.$ Define operators $ A_1, A_2 : \mathbb{X} \longrightarrow \mathbb{Y} $ as follows : 
\[ A_1(z) = \alpha  Tx_0, ~A_2(z) = Th.  \]
It is easy to verify that both $A_1,A_2 $ are bounded linear operators. Now $ \| T + \lambda A_1 \| \geq \| (T + \lambda A_1) x_n \| = \| Tx_n\| \rightarrow \|T\|$ so that $ T \bot_{B} A_1.$  Similarly $ \| T + \lambda A_2 \| \geq \| (T + \lambda A_2) x_0 \| = \| Tx_0 \| = \|T\|$ so that $ T \bot_{B} A_2.$ 
But  $ T = A_1 + A_2 $ which shows that $ T $ is not orthogonal to $ A_1 + A_2.$ This contradiction proves the result.
\end{proof}

We also have the following theorem, the proof of which follows in the same way as Theorem ~\ref{theorem:smooth compact necessary}.
\begin{theorem}\label{theorem:bounded unique}
Let $\mathbb{X}$ be a Banach space, $\mathbb{Y}$ be a normed  space and  $T \in B(\mathbb{X},\mathbb{Y})$ be a smooth point. If $ \| Tx_1 \| = \|Tx_2 \| = \|T\|$ for $ x_1, x_2 \in S_{\mathbb{X}} $ then $ x_1 = \pm x_2.$ 
\end{theorem}
Abatzoglou \cite{A} studied  the smoothness of bounded linear operators on a Hilbert space,  we here give an alternative proof of the same.

\begin{theorem}
Let $\mathbb{H}$ be a  Hilbert space. Then $T \in B(\mathbb{H})$ is a smooth point if and only if $T$ attains norm only at $\pm x_0$ and sup$\{\|Ty\|: x_0\bot y, y\in S_\mathbb{H}\} < \|T\|$.
\end{theorem}
\begin{proof} We first prove the necessary part in the following three steps : \\
(i) T attains norm at some point of $ S_{\mathbb{X}} $. \\
(ii) T attains norm at unique point $ x_0 \in S_{\mathbb{X}} $.\\
(iii)  {\em sup}$\{\|Ty\|: x_0\bot y, y\in S_\mathbb{H}\} < \|T\|$.

\smallskip
 Claim (i) : If $T$ does not attain its norm then exists a sequence $\{e_n\}$  of  orthonormal vectors such that $\|Te_n \| \longrightarrow \|T\|.$ Then there exists an orthonormal basis $\mathcal{B}$ containing $ \{e_n : n =1,2,\ldots\}.$ Define a linear operator $A_1$ on $H$ as : $ A_1e_{2n} = Te_{2n} $ and $A_1$ takes every other element of $\mathcal{B}$ to $0.$ We show that $A_1$ is bounded. Every element $z$ in $H$ can be written as $ z = x + y $ where $ x = \sum \alpha_{2n}e_{2n} $ and $ y \bot x.$ Now $ \|A_1 z \| = \| Tx \| \leq \|T\| \|x\| \leq \|T\| $ for every $z $ with $\|z\| =1.$ Thus $A_1$ is bounded. Consider another bounded operator $ A_2 = T - A_1.$ Next we claim that $ T \bot_{B} A_1 , T \bot_{B} A_2.$ Now $ \|Te_{2n+1}\| \longrightarrow \|T\| $ and $ \langle Te_{2n+1}, A_1 e_{2n+1} \rangle \longrightarrow 0 $ and so by Lemma 2 of \cite{P}, we get $ T \bot_{B}A_1.$ Similarly we can show that $ T \bot_{B}A_2.$ But $ T $ is not orthogonal to $A_1 + A_2 $  in the sense of Birkhoff-James- this contradicts the fact that $ T $is smooth. 

\smallskip
 Claim (ii) : Suppose $ \| Tx_1 \| = \| Tx_2 \| = \| T\| $ with $ x_1,x_2 \in S_{\mathbb{H}}, x_1 \neq \pm x_2.$ By Theorem 2.2 of Sain and Paul \cite{SP}, the norm attaining set $M_T $ is a unit sphere of some subspace of $\mathbb{H}.$ So without loss of generality we assume that $ x_1 \bot x_2.$ Let $ H_0 = \langle \{ x_1, x_2 \} \rangle.$ Then $ \mathbb{H} = H_0 \oplus H_0^{\bot} .$ Define $A_1, A_2 : \mathbb{H} \longrightarrow \mathbb{H} $ as follows : \\
$ A_1( c_1x_1 + c_2x_2 + h) = c_1 Tx_1, ~ A_2( c_1x_1 + c_2x_2 + h) = c_2 Tx_2 + Th$, where $ h \in H_0^{\bot}.$
Then as before it is easy to check that  both $A_1, A_2$ are bounded linear operators and $ T \bot_{B} A_1, T \bot_{B} A_2.$ But $ T = A_1 + A_2 $ and so $ T $ is not orthogonal to $ A_1 + A_2$ in the sense of Birkhoff-James, which contradicts the fact that $T$ is smooth.

\smallskip
 Claim (iii) Suppose $T$ attains norm only at $ \pm x_0 \in  S_\mathbb{H}\ $ and {\em sup}$\{\|Ty\|: x_0\bot y, y\in S_\mathbb{H}\} = \|T\|.$ Let $ H_0 = \langle \{ x_0\} \rangle .$ Then $ \mathbb{H} = H_0 \oplus H_0^{\bot} .$ Define $A_1,A_2 : \mathbb{H} \longrightarrow \mathbb{H} $ as follows: \\
Let $ z = x + y \in \mathbb{H} ,$ where $ x \in H_0, y \in  H_0^{\bot}. $ Then $ A_1 z = Tx, A_2 z = Ty.$ It is easy to check that both $A_1, A_2$ are bounded linear operators and $ T \bot_{B} A_1, T \bot_{B} A_2.$ But $ T = A_1 + A_2 $ and so $ T $ is not orthogonal to $ A_1 + A_2$ in the sense of Birkhoff-James, which contradicts the fact that $T$ is smooth.\\
\smallskip
We next prove the sufficient part. Assume $T$ attains norm only at $\pm x_0$ and sup$\{\|Ty\|: x_0\bot y, y\in S_\mathbb{H}\} < \|T\|$. Let $ T \bot_B A_i (i=1,2).$ Then by Theorem~\ref{theorem:Hilbert bounded}, $  Tx_0 \bot A_1x_0$ and $Tx_0 \bot A_2x_0.$ As $Tx_0 $ is a smooth point of $\mathbb{H}$ so $Tx_0 \bot (A_1+A_2)x_0.$ 
Then $ T\bot_B (A_1+A_2).  $ Thus $T$ is smooth.
\end{proof}
\begin{acknowledgement}\label{ackref}
We  thank Professor T.~K.~Mukherjee (Retd., Jadavpur University, India)
and Professor Abhijit Dasgupta (University of Detroit Mercy, USA)
for their invaluable suggestions while preparing this paper.
\end{acknowledgement}


\begin{thebibliography}{99}

\bibitem{A}  Abatzoglou, T. J.,  \textit{Norm derivatives on spaces of operators}, Math. Ann. \textbf{239} (1979), 129-135.


\bibitem{BS} Bhatia, R. and $\check{S}$emrl, P., \textit{Orthogonality of matrices and distance problem},  Linear Alg. Appl. \textbf{287},(1999) pp. 77-85.

\bibitem{B} Birkhoff, G.,  \textit{Orthogonality in linear metric spaces}, Duke Math. J., \textbf{1}, (1935) pp. 169-172.

\bibitem{DK} Deeb, W. and Khalil, R., \textit{Exposed and smooth points of some classes of operators in $L(l^p)$}, Journal of Functional Analysis, \textbf{103} (2) (1992), 217-228.

\bibitem{HR} Heinrich S., \textit{The differentiability of the norm in spaces of operators}, Functional. Anal, i Prilozen. (4) \textbf{9} (1975), 93-94. 
(English translation: Functional Anal. Appl. (4) 9 (1975), 360-362.)

\bibitem{HF} Hennefeld J., \textit{Smooth, compact operators}, Proc. Amer. Math. Soc., \textbf{77} (1)   (1979), 87-90.

\bibitem{HO} Holub, J.R.,  \textit{On the metric geometry of ideals of operators on Hilbert space}, Math. Ann. \textbf{201} (1973), 157-163.

\bibitem{J} James C. R., \textit{Orthogonality and linear functionals in normed linear spaces},  Transactions of the American Mathematical Society, Vol. \textbf{61} (1947) 265-292.

\bibitem{KY} Kittaneh, F. and Younis, R., \textit{Smooth points of certain operator spaces}, Int. Equations and Operator Theory \textbf{13} (1990), 849-855.

\bibitem{P} Paul, K., \textit{Translatable radii of an operator in the direction of another operator}, Scientiae Mathematicae  \textbf{2} (1999), 119-122. 

\bibitem{SP} Sain, D.  and Paul, K.,   \textit{ Operator norm attainment and inner product spaces},   Linear Algebra and its Applications,   \textbf{439} (2013) 2448--2452.

\bibitem{SPH} Sain, D., Paul, K. and Hait, S., \textit{ Operator norm attainment and Birkhoff-James orthogonality},  Linear Algebra and its Applications, \textbf{476} (2015) 85--97.

\bibitem{W} Werner, W., \textit{Smooth points in some spaces of bounded operators}, Integral Equation Operator Theory, \textbf{15} (1992) 496-502. 
\end{thebibliography}
\end{document}